\documentclass[12pt,reqno,a4paper,twoside]{article}
\usepackage{amsmath,amsthm,amstext,amscd,amssymb,euscript}
\usepackage{epsf}
\renewcommand{\phi}{\varphi}

\renewcommand{\subset}{\subseteq}
\renewcommand{\emptyset}{\varnothing}

\def\1{ {\mathit{1} \!\!\>\!\! I} }

\makeatletter
\@addtoreset{equation}{section}
\makeatother

\newtheorem{ittheorem}{Theorem}
\newtheorem{itlemma}{Lemma}
\newtheorem{itcorollary}{Corollary}
\newtheorem{itproposition}{Proposition}
\newtheorem{itdefinition}{Definition}
\newtheorem{itremark}{Remark}

\newenvironment{theorem}{\addtocounter{equation}{1}
\begin{ittheorem}}{\end{ittheorem}}

\newenvironment{lemma}{\addtocounter{equation}{1}
\begin{itlemma}}{\end{itlemma}}
\newenvironment{corollary}{\addtocounter{equation}{1}
\begin{itcorollary}}{\end{itcorollary}}

\newenvironment{proposition}{\addtocounter{equation}{1}
\begin{itproposition}}{\end{itproposition}}

\newenvironment{definition}{\addtocounter{equation}{1}
\begin{itdefinition}}{\end{itdefinition}}

\newenvironment{remark}{\addtocounter{equation}{1}
\begin{itremark}}{\end{itremark}}

\newcommand{\beq}{\begin{eqnarray}}
\newcommand{\eeq}{\end{eqnarray}}

\newcommand{\be}{\begin{equation}}
\newcommand{\ee}{\end{equation}}

\newcommand{\bl}{\begin{lemma}}
\newcommand{\el}{\end{lemma}}

\newcommand{\br}{\begin{remark}}
\newcommand{\er}{\end{remark}}

\newcommand{\bt}{\begin{theorem}}
\newcommand{\et}{\end{theorem}}

\newcommand{\bd}{\begin{definition}}
\newcommand{\ed}{\end{definition}}

\newcommand{\bp}{\begin{proposition}}
\newcommand{\ep}{\end{proposition}}

\newcommand{\bc}{\begin{corollary}}
\newcommand{\ec}{\end{corollary}}

\newcommand{\bpr}{\begin{proof}}
\newcommand{\epr}{\end{proof}}

\newcommand{\bi}{\begin{itemize}}
\newcommand{\ei}{\end{itemize}}

\newcommand{\ben}{\begin{enumerate}}
\newcommand{\een}{\end{enumerate}}

\newcommand{\Z}{\mathbb Z}
\newcommand{\R}{\mathbb R}

\newcommand{\pee}{\ensuremath{\mathbb{P}}}

\newcommand{\vi}{\ensuremath{\varphi}}

\newcommand{\la}{\ensuremath{\Lambda}}
\newcommand{\si}{\ensuremath{\sigma}}

\def\now{
\ifnum\time<60
          12:\ifnum\time<10 0\fi\number\time am
          \else
            \ifnum\time>719\chardef\a=`p\else\chardef\a=`a\fi
          \hour=\time
          \minute=\time
          \divide\hour by 60 
          \ifnum\hour>12\advance\hour by -12\advance\minute by-720 \fi
          \number\hour:%
          \multiply\hour by 60 
          \advance\minute by -\hour
          \ifnum\minute<10 0\fi\number\minute\a m\fi}
\newcount\hour
\newcount\minute
\numberwithin{equation}{section}         
\newtheorem{thm}{Theorem}[section]
\newtheorem{lem}[thm]{Lemma}
\newtheorem{defn}[thm]{Definition}
\theoremstyle{remark}

\newcommand{\caH}{{\mathcal H}}

\begin{document}
\title{{\bf Transformations of one-dimensional Gibbs measures with infinite range interaction}}
\author{
F.\ Redig \footnote{University of Nijmegen, IMAPP, Heyendaalse weg 135, 6525 AJ NIjmegen,
The Netherlands
{\em redig@math.leidenuniv.nl}}\\
F.\ Wang\footnote{Mathematisch Instituut Universiteit Leiden, Niels
Bohrweg 1, 2333 CA Leiden, The Netherlands, {\em
wangf@math.leidenuniv.nl}} } \maketitle
\begin{quote}
{\bf Abstract:} We study single-site stochastic and deterministic
transformations of one-dimensional Gibbs measures in the uniqueness
regime with infinite-range interactions.
We prove conservation of Gibbsianness and give quantitative estimates
on the decay of the transformed potential. As examples,
we consider exponentially decaying potentials, and potentials decaying
as a power-law.
\end{quote}
\vspace{12pt}

\section{Introduction}\label{1}
Local transformations of Gibbs measures can be non-Gibbs. In
\cite{enter}, the mechanism behind the creation of non-Gibbsianness is explained as a hidden
phase transition: conditioned on a certain configuration of the transformed
spins, the original spins can exhibit a phase transition. Even if the
untransformed system is not in a phase transition regime, by conditioning
on the transformed configuration we can bring it into a regime of phase transition.
In a regime of strong uniqueness, such as the Dobrushin uniqueness regime,
or the complete analyticity regime, one expects that Gibbs measures turn into
Gibbs measures under stochastic or deterministic disjoint-block transformations.

For one-dimensional systems in the uniqueness regime, one also expects that local transformations
conserve the Gibbs property. Using disagreement percolation, this has been
proved for finite-range potentials, \cite{mvdv}. The technique of disagreement
percolation has however not been extended to the case of infinite range interactions,
and in fact (at present) breaks down in that context. Further, it is also known that
in the uniqueness regime in dimension one, decimating sufficiently many times
brings the system into a regime where cluster expansion can be obtained, and
hence the system becomes completely analytic \cite{oli}. Finally, in the context of
dyamical systems, it has been shown recently \cite{chazo} that a Gibbs measure with an exponentially
decaying interaction
transforms into a Gibbs measure with an interaction that decays
at least as a stretched exponential under a transformation that ``confuses''
symbols (i.e., the transformed spin is determined by a partition of the untransformed spin).

In this paper we consider lattice spin systems in one dimension, with an interaction
that is allowed to be of infinite range. We consider single-site stochastic and deterministic transformations.
We prove that under a uniqueness condition (see \ref{potential-assumption2} below), the transformed
measure is Gibbs.
We further prove that, if the initial interaction is exponentially decaying, then the transformed
interaction decays exponentially as well. If the initial interaction decays (in some sense) as
a power law with power $\alpha$ (which is chosen big enough to be in the uniqueness regime),
then the tranformed interaction can be estimated with a (smaller) power as well.

The method of proof is based on two ingredients. One ingredient is
classical: the single-site conditional probabilities of the
transformed measure can be written as the expected value of a local
function in a Gibbs measure that depends on the conditioning. The
dependence on the conditioning, in the case of a single-site
transformation is in the form of a spatially varying magnetic field.
The second step is to control how the local function expectation
depends on this magnetic field. This reduces to the problem of how
well a local expectation is approximated by finite-volume Gibbs
measure expectations (in a context which is not spatially
homogeneous because of the presence of the magnetic field depending
on the conditioning). In this second step we use coupling, in the
spirit of \cite{Bressaud99}. As a consequence of this method, we
obtain, besides Gibbsianness, estimates on the decay of the
transformed potential (where we use the so-called Kozlov potential
defined on lattice intervals).

Our paper is organized as follows: we start with basic definitions on Gibbs measures,
potentials, and define the transformations that we consider.
Section \ref{2} is devoted to the case of stochastic single-site transformations.
Section \ref{3} contains the single-site deterministic case.

\section{Gibbs measures and their transformations}\label{2}

\subsection{One-dimensional Gibbs measures}

We consider lattice spin systems, with
configuration $\Omega=S^{\mathbb{Z}}$,
where $S$, the single-site space, is a finite set. We equip $\Omega$ with the product topology.
The
set of all finite subsets of $\mathbb{Z}$ is denoted by
$\mathcal{L}$. For $\Lambda\in\mathcal{L}$ and $\sigma\in\Omega$, we
denote by $\sigma_{\Lambda}$ the restriction of $\sigma$ to
$\Lambda$, while $\Omega_{\Lambda}$ denotes the set of all such
restrictions.

A function $f:\Omega\rightarrow\mathbb{R}$ is called \emph{local} if
there exists a finite set $\Delta\subset\mathbb{Z}$ such that
$f(\eta)=f(\sigma)$ for $\eta$ and $\sigma$
coinciding on $\Delta$.

Continuity in the product topology coincides with quasi-locality, i.e.,
a function $f:\Omega\to\R$ is continuous if and only if it is a uniform
limit of local functions, more precisely if
\begin{equation} \label{conti}
\lim_{\Lambda\uparrow\mathbb{Z}}\sup_{\xi,\zeta\in\Omega}
\left|f(\omega_{\Lambda}\xi_{\Lambda^c})-f(\omega_{\Lambda}\zeta_{\Lambda^c})\right|=0,
\end{equation}

\begin{defn} A function
$\Phi:\mathcal{L}\times\Omega\rightarrow\mathbb{R}$ such that
$\Phi(A,\sigma)$ depends only on $\sigma(x)$, $x\in A$ for $\forall
A\in \mathcal{L}$, is called a \textit{\textbf{potential}}. A
potential is \textit{\textbf{uniformly absolutely convergent}} if
for all $x\in\mathbb{Z}$
\begin{equation}
\sum_{A\ni x}\left\|\Phi(A,\sigma)\right\|_{\infty}<\infty,
\end{equation}
where
$\|\Phi(A,\sigma)\|_{\infty}=\sup_{\sigma\in\Omega}|\Phi(A,\sigma)|$.
\end{defn}
For $\Phi\in \mathcal{B}$, $\zeta\in \Omega$,
$\Lambda\in\mathcal{L}$, we define the finite-volume
\textit{Hamiltonian} with boundary condition $\zeta$ as
\begin{equation}
H_{\Lambda}^{\zeta}(\sigma)=\sum_{A\bigcap\Lambda\neq\emptyset}\Phi(A,\sigma_{\Lambda}\zeta_{\Lambda^c}).
\end{equation}
Corresponding to this Hamiltonian we have the \textit{finite-volume
Gibbs measures} $\mu_{\Lambda}^{\Phi,\zeta}$,
$\Lambda\in\mathcal{L}$, with boundary condition $\zeta$, defined on $\Omega$ by
\begin{equation} \label{finiteGibbs}
\int f(\xi)\mu_{\Lambda}^{\Phi,\zeta}(d\xi)=
\sum_{\sigma_{\Lambda}\in\Omega_{\Lambda}}
f(\sigma_{\Lambda}\zeta_{\Lambda^c})\frac{\exp\left(-H_{\Lambda}^{\zeta}(\sigma)\right)}{Z_{\Lambda}^{\zeta}},
\end{equation}
where $Z_{\Lambda}^{\zeta}$ denotes the partition function
normalizing $\mu_{\Lambda}^{\Phi,\zeta}$ to a probability measure
and $f:\Omega\mapsto \mathbb{R}$ denotes any local function. For a
probability measure $\mu$ on $\Omega$, we denote by
$\mu_{\Lambda}^{\zeta}$ the condition probability distribution of
$\sigma(x)$, $x\in\Lambda$, given
$\sigma_{\Lambda^c}=\zeta_{\Lambda^c}$, which is of course only
$\mu-a.s.$ defined.
\begin{defn} For $\Phi\in\mathcal{B}$, we call $\mu$ a \textit{\textbf{Gibbs measure}}
with potential $\Phi$ if a version of its conditional probabilities
coincides with the ones prescribed in (\ref{finiteGibbs}), i.e., if
\begin{equation}
\mu_{\Lambda}^{\Phi,\zeta}=\mu_{\Lambda}^{\zeta} \;\;\;\;\; \mu -
a.s. \;\;\;\;\; \forall\Lambda\in\mathcal{L},\zeta\in\Omega .
\end{equation}
\end{defn}
We assume that the potential $\Phi$ satisfies
the following condition.

\begin{equation}\label{potential-assumption}
\sup_{i\in \mathbb{Z}}\sum_{A\ni i,diam(A)\geq
K}\|\Phi(A,\sigma)\|_{\infty} = f(K)
\end{equation}
where $f$ satisfies
\begin{equation}\label{potential-assumptionKAK}
\sum_{n=0}^\infty  f(n) <\infty
\end{equation}
Under the condition \ref{potential-assumptionKAK}, the potential
$\Phi$ admits only one Gibbs measure $\mu=\mu_{\Phi}$, see
\cite{Georgii88}, section 8.3. Condition
\ref{potential-assumptionKAK} of course implies
\begin{equation} \label{potential-assumption2}
\lim_{k\to\infty}\sum_{j\geq 0} f(j+k) =0.
\end{equation}
We abbreviate
\be\label{aam}
F_k= \sum_{j\geq 0} 2f(j+k)
\ee
Remark that in the case of a translation invariant potential, the supremum
in \eqref{potential-assumption} can be omitted and then
\[
 f(K) = \sum_{A\ni 0,diam(A)\geq
K}\|\Phi(A,\sigma)\|_{\infty}
\]

\begin{defn} A version of conditional probabilities
$\{\mu(\cdot|\zeta_{\Lambda^c}):
\zeta_{\Lambda^c}\in\Omega_{\Lambda^c},\Lambda\in\mathcal{L}\}$ is
called \textit{\textbf{uniformly non null}} if for every
$\Lambda\in\mathcal{L}$, there exists a constant $m_{\Lambda}>0$
such that for every $\omega\in\Omega$
\begin{equation}
\mu_{\Lambda}^{\omega}(\omega)\geq m_{\Lambda}.
\end{equation}
\end{defn}
The following theorem due to Kozlov \cite{Kozlov74} and Sullivan
\cite{Sullivan73} gives a criterion to decide whether a given
measure is Gibbsian.
\begin{thm} \label{Kozlov-Sullivan}
A probability measure $\mu$ on $(\Omega,\mathcal{F})$ is a Gibbs
measure with respect to a uniformly absolutely convergent potential
iff there exists a version of its conditional probabilities that is
continuous and uniformly non null.
\end{thm}
\br
Theorem \ref{Kozlov-Sullivan} is constructive, i.e.,
the potential is constructed from the conditional probabilities. See
section \ref{4} for the explicit form. In our one-dimensional
case, it is non-vanishing on lattice intervals only, i.e., sets
of the form $[i,j] = \{ i,i+1,\ldots,j\}$. Therefore, if we start from a Gibbs
measure we can assume without loss of generality that the potential is
non-zero only on lattice intervals.
\er

\subsection{Transformations of Gibbs measures}
We consider two types of transformation: single-site stochastic tranformations and
single-site deterministic transformations.

We first consider
single site stochastic transformation, i.e., for a given $\sigma$,
the distribution of the image spin configuration is a product
measure on $(S')^{\Z}$
\begin{equation}
T(\xi|\sigma)=\prod_{i\in\mathbb{Z}}P_i(\xi_i|\sigma_i).
\end{equation}
Here, $S'$ denotes the alphabet of the image-spin, and satisfies
$|S'|\leq |S|$.

We assume that the transition kernel of a single site is strictly positive.
That is, for $i\in\mathbb{Z}$,
\begin{equation}\label{kernel-assumption}
\inf_{i\in \Z, \xi_i, \si_i} P_i(\xi_i|\sigma_i) > 0.
\end{equation}
The distribution of the image spin is then defined as
\begin{equation}
\mu\circ T(d\xi)=\int T(d\xi|\sigma) \mu(d\sigma).
\end{equation}

The second case is a single-site deterministic transformation $T:\Omega\to\Omega'$
induced by a map $\vi: S\to S'$ given by
\be\label{dettra1}
(T(\si))_i=:\si'_i = \vi(\si_i)
\ee

\section{Stochastic single-site transformations}\label{3}

\begin{thm} \label{thoerem1}
For single site stochastic transformations, if the potential $\Phi$
corresponding to the initial Gibbs measure $\mu$ satisfies condition
(\ref{potential-assumption}), \eqref{potential-assumptionKAK}, then
the transformed measure $\mu\circ T$ is a Gibbs measure.
\end{thm}

\begin{proof}
First of all, $\{\mu\circ T(\cdot|\zeta_{\Lambda^c}):
\zeta_{\Lambda^c}\in\Omega_{\Lambda^c},\Lambda\in\mathcal{L}\}$ is
uniformly non null thanks to the positivity assumption of a single
site's transformation kernel in (\ref{kernel-assumption}). We then
proceed with the proof in two steps.

First, we express the one-site
conditional probabilities $\mu\circ
T(\xi_{0}|\xi_{\mathbb{Z}\setminus \{0\}})$ as averages of a local
observable over a Gibbs measure depending on the conditioning $\xi$.
This is in the spirit of \cite{opoku}, but simpler since the transformation
is stochastic, and hence the ``constrained first layer model''
of \cite{opoku} is
``not constrained'' (given the image configuration, all configurations
are possible as originals).

Second, we use a ``house-of-cards" coupling technique (see
\eqref{pig}) in the spirit of \cite{Bressaud99} to prove the
dependence of this local expectation on the conditioning $\xi$. We
restrict to the conditional expectation of the transformed spin at
the origin, given the transformed spins outside the origin. The same
argument applies to conditional expectation of the spin at any other
site.

\emph{Step 1.}
\begin{eqnarray}\label{fuut}
\nonumber \mu\circ T(\xi_{0}|\xi_{\mathbb{Z}\setminus
\{0\}})&=&\lim_{\Lambda\uparrow\Z}\frac{\mu_{\Lambda}\circ
T(\xi)}{\sum_{\widetilde{\xi}_{0}}\mu_{\Lambda}\circ
T(\widetilde{\xi}_{0} \xi_{\Lambda\setminus \{0\}})}\\\nonumber
&=&\lim_{\Lambda\uparrow\Z}\frac{\sum_{\sigma_{\Lambda}}\mu_{\Lambda}\circ
T(\xi|\sigma_{\Lambda})\mu_{\Lambda}(\sigma_{\Lambda})}
{\sum_{\sigma_{\Lambda}}\sum_{\widetilde{\xi}_{0}}\mu_{\Lambda}\circ
T(\widetilde{\xi}_0\xi_{\Lambda\setminus
\{0\}}|\sigma_{\Lambda})\mu_{\Lambda}(\sigma_{\Lambda})} \\\nonumber
&=& \lim_{\Lambda\uparrow\Z}
\frac{\sum_{\sigma_{\Lambda}}\prod_{i}P_i(\xi_i|\sigma_i)\mu_{\Lambda}(\sigma_{\Lambda})}
{\sum_{\sigma_{\Lambda}}\prod_{i\neq0}
P_i(\xi_i|\sigma_i)\mu_{\Lambda}(\sigma_{\Lambda})}\\\nonumber &=&
\lim_{\Lambda\uparrow\Z}
\frac{\sum_{\sigma_{\Lambda}}\prod_{i}P_i(\xi_i|\sigma_i)\mu_{\Lambda}(\sigma_{\Lambda})}
{\sum_{\sigma_{\Lambda}}\prod_{i}
P_i(\xi_i|\sigma_i)\mu_{\Lambda}(\sigma_{\Lambda})\frac{1}{P_0(\xi_0|\sigma_0)}}
\\
&=&
\left(\mu^{\xi}\left(\frac{1}{P_0\left(\xi_0|\sigma_0\right)}\right)\right)^{-1}
\label{functional-form}
\end{eqnarray}


where $\mu^{\xi}$ is a new Gibbs measure with potential
\begin{eqnarray}\label{Gibbs-xi}
\Phi_A^{\xi}(\sigma)=
\begin{cases} \Phi_A(\sigma) & \mbox{if $|A|>1$}\\
\Phi_A(\sigma) - \log{P_i(\xi_i|\sigma_i)}  &     \mbox{if
$A=\{i\}$}\end{cases}.
\end{eqnarray}
and where the expectation
\[
\left(\mu^{\xi}\left(\frac{1}{P_0\left(\xi_0|\sigma_0\right)}\right)\right)^{-1}
\]
in \eqref{fuut}
is w.r.t.\ $\sigma_0$, with fixed $\xi$.
Remark that this Gibbs measure is uniquely defined, because
it is a single-site modification of the original potential
$\Phi$, for which we have uniqueness by condition \ref{potential-assumption2}.
The equalities in
(\ref{functional-form}) are almost surely with respect to $\mu\circ
T$. Therefore, it suffices to show that
$\mu^{\xi}(P_0^{-1}(\xi_0|\sigma_0))$ is continuous as a function of
$\xi$. Indeed, this then implies that $\mu\circ
T(\xi_{0}|\xi_{\mathbb{Z}\setminus \{0\}})$ admits a version that is
continuous as a function of $\xi$, which implies Gibbsianness, by
Theorem \ref{Kozlov-Sullivan}. The problem boils down to proving
(cf. \eqref{conti})
\begin{equation}\nonumber
\lim_{\Lambda\uparrow\mathbb{Z}}
\left|\mu^{\xi_{\Lambda}\eta_{\Lambda^c}}(P_0^{-1}(\xi_0|\sigma_0))-
\mu^{\xi_{\Lambda}\eta'_{\Lambda^c}}(P_0^{-1}(\xi_0|\sigma_0))\right|=0.
\end{equation}
The form of the potential of $\mu^{\xi_{\Lambda}\eta_{\Lambda^c}}$,
given in (\ref{Gibbs-xi}), implies that the Hamiltonian of the
corresponding finite-volume Gibbs measure
$\mu^{\xi_{\Lambda}\eta_{\Lambda^c}}_{\Lambda,\zeta}$ with boundary
condition $\zeta$ has the following form
\begin{equation} \nonumber
H^{\xi_{\Lambda}\eta_{\Lambda^c}}_{\Lambda,\zeta}(\sigma) =
H_{\Lambda,\zeta}(\sigma) - \sum_{i\in\Lambda}\log
P_i(\xi_i|\sigma_i).
\end{equation}
Hence $\mu^{\xi_{\Lambda}\eta_{\Lambda^c}}_{\Lambda,\zeta}$ is
independent of $\eta$ and denoted as
$\mu^{\xi_{\Lambda}}_{\Lambda,\zeta}$, which implies that
\begin{eqnarray}\nonumber
\left|\mu^{\xi_{\Lambda}\eta_{\Lambda^c}}(P_0^{-1}(\xi_0|\sigma_0))-\mu^{\xi_{\Lambda}\eta'_{\Lambda^c}}(P_0^{-1}(\xi_0|\sigma_0))\right|
\leq
\;\;\;\;\;\;\;\;\;\;\;\;\;\;\;\;\;\;\;\;\;\;\;\;\;\;\;\;\;\;\;\;\;\;\;\;\;\;\;\;\;\;\;\;\;\;\;\;\;\;
\\\nonumber
\left|\mu^{\xi_{\Lambda}\eta_{\Lambda^c}}(P_0^{-1}(\xi_0|\sigma_0))-\mu^{\xi_{\Lambda}\eta_{\Lambda^c}}_{\Lambda,\zeta}
(P_0^{-1}(\xi_0|\sigma_0))\right| +
\left|\mu^{\xi_{\Lambda}\eta'_{\Lambda^c}}_{\Lambda,\zeta}(P_0^{-1}(\xi_0|\sigma_0))
-\mu^{\xi_{\Lambda}\eta'_{\Lambda^c}}(P_0^{-1}(\xi_0|\sigma_0))\right|.
\end{eqnarray}

At this stage, it suffices to prove that, uniformly in $\zeta$,
\begin{equation}\label{Gibbs-conditioning}
\mu^{\xi}(P_0^{-1}(\xi_0|\sigma_0))
=\lim_{\Lambda\uparrow\mathbb{Z}}\mu^{\xi_{\Lambda}}_{\Lambda,\zeta}(P_0^{-1}(\xi_0|\sigma_0)).
\end{equation}

\emph{Step 2.} It is sufficient for (\ref{Gibbs-conditioning})
if we can prove
\begin{equation}\nonumber
\mu^{\widetilde{\Phi}}(\sigma_0)=\lim_{\Lambda\uparrow\mathbb{Z}}
\mu^{\widetilde{\Phi}}_{\Lambda,\zeta}(\sigma_0),
\end{equation}
where $\widetilde{\Phi}$ is a general potential satisfying condition
(\ref{potential-assumptionKAK}). More precisely, we will prove that
\begin{equation}\label{cromboo}
\lim_{l\rightarrow\infty}\sup_{\zeta,\zeta'}\left|\mu^{\widetilde{\Phi}}_{[-l,l],\zeta}
(\sigma_0)-\mu^{\widetilde{\Phi}}_{[-l,l],\zeta'}(\sigma_0)\right|=0,
\end{equation}
where $\mu^{\widetilde{\Phi}}_{[-l,l],\zeta}$ means the measure for
configurations on $[-l,l]$ conditioned on the boundary
$\zeta_{[-l,l]^c}$. For simplicity, we will omit the superscript $
\widetilde{\Phi}$ hereafter. The speed of this convergence to zero
(as a function of $l$) will determine the decay of the potential
associated to the transformed measure (see later).

To prove \eqref{cromboo}, we couple the measures
$\mu_{[-l,l],\zeta}(\sigma_{[-l,l]}=\cdot)$ and
$\mu_{[-l,l],\zeta'}(\sigma_{[-l,l]}=\cdot)$, i.e., we construct a
probability measure on pairs $(\sigma^1_{[-l,l]},
\sigma^2_{[-l,l]})$ with marginals
$\mu_{[-l,l],\zeta}(\sigma_{[-l,l]}=\cdot)$ and
$\mu_{[-l,l],\zeta'}(\sigma_{[-l,l]}=\cdot)$. The construction of
the coupling follows an iterative procedure (inspired by
\cite{Georgii01}, Section 7), where we generate in every stage a
pair of two spins corresponding to the interior boundary spins at
that stage. Initially, we generate $(\sigma_{-l}^1,\sigma_{l}^1)$
and $(\sigma_{-l}^2,\sigma_{l}^2)$ according to the maximal
coupling\footnote{For details of coupling and maximal coupling, we
refer to \cite{Thorisson00}.} of
$\mu_{[-l,l],\zeta}(\sigma_{-l}=\cdot,\sigma_{l}=\cdot)$ and
$\mu_{[-l,l],\zeta'}(\sigma_{-l}=\cdot,\sigma_{l}=\cdot)$. Having
generated
$(\sigma_{-l}^i,\sigma_{l}^i)$,$(\sigma_{-l+1}^i,\sigma_{l-1}^i)$
$\ldots$ $(\sigma_{-l+m}^i,\sigma_{l-m}^i)$, for $i=1,2$, we
generate $(\sigma_{-l+m+1}^1,\sigma_{l-m-1}^1)$ and
$(\sigma_{-l+m+1}^2,\sigma_{l-m-1}^2)$ according to the maximal
coupling of
\[
\mu_{[-l+m+1,l-m-1],\zeta\sigma_{[-l+1,-l+m]\bigcup[l-1,l-m]}^1}(\sigma_{-l+m+1}=\cdot,\sigma_{l-m-1}=\cdot)
\]
and
\[
\mu_{[-l+m+1,l-m-1],\zeta'\sigma^2_{[-l+1,-l+m]\bigcup[l-1,l-m]}}(\sigma_{-l+m+1}=\cdot,\sigma_{l-m-1}=\cdot).
\]
To estimate
$|\mu_{[-l,l],\zeta}(\sigma_0)-\mu_{[-l,l],\zeta'}(\sigma_0)|$, we
use the coupling just described, and proceed as in a "house-of-cards
coupling" method of Bressaud-Fern\'{a}ndez-Galves \cite{Bressaud99}.
When we generate the symbols $\sigma_{-l+k},\sigma_{l-k}$, we think
of this as being at time instant $k$ in the coupling. Suppose that
for the last $m$ time instants in the coupling, we had matches, then
as in \cite{Bressaud99} we have to estimate the probability of a
mismatch at time instant $m+1$. This is done in the next lemma.

\begin{lem} For $-l<-n_2<-n_1\leq 0 \leq n_1<n_2<l$, $n_2-n_1= m$,
let $\xi$ and $\zeta$ be two configurations on the complement of
$[-n_1,n_1]$ such that they agree on
$\Delta_m=[-n_2,-n_1-1]\bigcup[n_1+1,n_2]$, then

\begin{eqnarray}
\sup_{\alpha, \beta,\zeta,\xi,\zeta_{\Delta_m}=\xi_{\Delta_m}}
|\mu_{[-n_1,n_1],\zeta}(\sigma_{-n_1}=\alpha,\sigma_{n_1}=\beta) -
\mu_{[-n_1,n_1],\xi}(\sigma_{-n_1}=\alpha,\sigma_{n_1}=\beta)|
\nonumber
\leq 2(e^{F_m}-1),
\end{eqnarray}
where
$F_m$ is defined in \eqref{aam}.
\end{lem}
\begin{proof} Start with
\begin{equation}\nonumber
\mu_{[-n_1,n_1],\zeta}(\sigma_{-n_1}=\alpha,\sigma_{n_1}=\beta) =
\sum_{\sigma'}\frac{e^{-H_{[-n_1,n_1]}^{\zeta}(\alpha\sigma'\beta)}}{Z_{[-n_1,n_1]}^{\zeta}},
\end{equation}
where we abbreviated $\alpha\sigma'\beta$ to be the configuration
$\sigma_{[-n_1,n_1]}$ with $\sigma_{-n_1}=\alpha$,
$\sigma_{n_1}=\beta$ and $\sigma_{[-n_1+1,n_1-1]}=\sigma'$, and
where the sum runs over all configurations $\sigma'$ on
$[-n_1+1,n_1-1]$.
We then proceed as follows:
\begin{eqnarray}\nonumber
&&\;\;\;\;\sup_{\alpha\sigma'\beta,\zeta,\xi,\zeta_{\Delta_m}=\xi_{\Delta_m}}\left|\mu_{[-n_1,n_1],\zeta}(\sigma_{-n_1}=\alpha,\sigma_{n_1}=\beta)
-
\mu_{[-n_1,n_1],\xi}(\sigma_{-n_1}=\alpha,\sigma_{n_1}=\beta)\right|\\
\nonumber &&=
\sup_{\alpha\sigma'\beta,\zeta,\xi,\zeta_{\Delta_m}=\xi_{\Delta_m}}\left|\sum_{\sigma'}\frac{e^{-H_{[-n_1,n_1]}^{\zeta}(\alpha\sigma'\beta)}}{Z_{[-n_1,n_1]}^{\zeta}}-
\sum_{\sigma'}\frac{e^{-H_{[-n_1,n_1]}^{\xi}(\alpha\sigma'\beta)}}{Z_{[-n_1,n_1]}^{\xi}}\right|
\\\nonumber
&&\leq
\sup_{\alpha\sigma'\beta,\zeta,\xi,\zeta_{\Delta_m}=\xi_{\Delta_m}}
\left\{\left|\frac{\sum_{\sigma'}e^{-H_{[-n_1,n_1]}^{\zeta}(\alpha\sigma'\beta)}}{Z_{[-n_1,n_1]}^{\zeta}}
-\frac{\sum_{\sigma'}e^{-H_{[-n_1,n_1]}^{\xi}(\alpha\sigma'\beta)}}{Z_{[-n_1,n_1]}^{\zeta}}\right|
\right.\\\nonumber &&
\;\;\;\;\;\;\;\;\;\;\;\;\;\;\;\;\;\;\;\;\;\;\;\;\;\;\;\;\;\;\;\;\;\;\;
+\left.\left|\frac{\sum_{\sigma'}e^{-H_{[-n_1,n_1]}^{\xi}(\alpha\sigma'\beta)}}{Z_{[-n_1,n_1]}^{\zeta}}
-\frac{\sum_{\sigma'}e^{-H_{[-n_1,n_1]}^{\xi}(\alpha\sigma'\beta)}}{Z_{[-n_1,n_1]}^{\xi}}\right|\right\}
\\\nonumber
&&\leq
\sup_{\alpha\sigma'\beta,\zeta,\xi,\zeta_{\Delta_m}=\xi_{\Delta_m}}
\left\{\left|\frac{\sum_{\sigma'}e^{-
H_{[-n_1,n_1]}^{\zeta}(\alpha\sigma'\beta)}}{\sum_{\sigma'}e^{-H_{[-n_1,n_1]}^{\xi}(\alpha\sigma'\beta)}}-1\right|
+
\left|\frac{Z_{[-n_1,n_1]}^{\xi}}{Z_{[-n_1,n_1]}^{\zeta}}-1\right|\right\}
\\\nonumber
&&=
\sup_{\alpha\sigma'\beta,\zeta,\xi,\zeta_{\Delta_m}=\xi_{\Delta_m}}
\left\{\left|\frac{\sum_{\sigma'}e^{-
H_{[-n_1,n_1]}^{\zeta}(\alpha\sigma'\beta)}}{\sum_{\sigma'}e^{-H_{[-n_1,n_1]}^{\xi}(\alpha\sigma'\beta)}}-1\right|
+ \left|\frac{\sum_{\alpha'\sigma'\beta'}e^{-
H_{[-n_1,n_1]}^{\zeta}(\alpha'\sigma'\beta')}}
{\sum_{\alpha'\sigma'\beta'}e^{-H_{[-n_1,n_1]}^{\xi}(\alpha'\sigma'\beta')}}-1\right|\right\},
\end{eqnarray}
where the sums in the second fraction run over all configuration
$\alpha'\sigma'\beta'$ on $[-n_1,n_1]$. By using the elementary
inequalities $\min_{i\in\{1,\ldots,n\}}\frac{a_i}{b_i} \leq
\frac{\sum_{i=1}^n a_i}{\sum_{i=1}^nb_i} \leq
\max_{i\in\{1,\ldots,n\}}\frac{a_i}{b_i}$ and $|e^x-1|\leq
e^{|x|}-1$, we obtain

\begin{eqnarray}\nonumber
&&\sup_{\alpha\sigma'\beta,\zeta,\xi,\zeta_{\Delta_m}=\xi_{\Delta_m}}
\left\{\left|\frac{\sum_{\sigma'}e^{-
H_{[-n_1,n_1]}^{\zeta}(\alpha\sigma'\beta)}}{\sum_{\sigma'}e^{-H_{[-n_1,n_1]}^{\xi}(\alpha\sigma'\beta)}}-1\right|
+ \left|\frac{\sum_{\alpha'\sigma'\beta'}e^{-
H_{[-n_1,n_1]}^{\zeta}(\alpha'\sigma'\beta')}}
{\sum_{\alpha'\sigma'\beta'}e^{-H_{[-n_1,n_1]}^{\xi}(\alpha'\sigma'\beta')}}-1\right|\right\}
\\\nonumber && \leq \sup_{\alpha\sigma'\beta,\zeta,\xi,\zeta_{\Delta_m}=\xi_{\Delta_m}}
\left\{\left(
e^{\sup_{\sigma'}\left|H_{[-n_1,n_1]}^{\xi}(\alpha\sigma'\beta)-
H_{[-n_1,n_1]}^{\zeta}(\alpha\sigma'\beta)\right|}-1\right) \right.
\\\nonumber &&
\;\;\;\;\;\;\;\;\;\;\;\;\;\;\;\;\;\;\;\;\;\;\;\;\;\;\;\;\;\;\;\;\;\;\;
\left.+ \left(
e^{\sup_{\alpha'\sigma'\beta'}\left|H_{[-n_1,n_1]}^{\xi}(\alpha'\sigma'\beta')-
H_{[-n_1,n_1]}^{\zeta}(\alpha'\sigma'\beta')\right|}-1\right)\right\}
\\\nonumber
&& =
2\left(e^{\sup_{\zeta,\xi,\zeta_{\Delta_m}=\xi_{\Delta_m},\alpha\sigma'\beta}\left|H_{[-n_1,n_1]}^{\xi}(\alpha\sigma'\beta)-
H_{[-n_1,n_1]}^{\zeta}(\alpha\sigma'\beta)\right|}-1\right),
\end{eqnarray}
Now
\begin{eqnarray} \nonumber &&
\sup_{\alpha\sigma'\beta,\zeta,\xi,\zeta_{\Delta_m}=\xi_{\Delta_m}}\left|H_{[-n_1,n_1]}^{\xi}(\alpha\sigma'\beta)-
H_{[-n_1,n_1]}^{\zeta}(\alpha\sigma'\beta)\right|
\\
\nonumber &\leq& \sum_{j=-n_1}^{n_1}\sum_{A\ni j, diam(A)\geq
(n_2-j)\bigwedge(j-(-n_2))} 2\parallel \Phi(A,\sigma)
\parallel_{\infty} \\\nonumber
&\leq& \sum_{k=0}^{2n_1}\sup_{j\in \Z}\sum_{A \ni j, diam(A)\geq
k+m}2\parallel \Phi(A,\sigma)
\parallel_{\infty} \\ \nonumber &\leq&
\sum_{k=0}^{\infty} 2f(k+m) = F_m.
\end{eqnarray}
(Recall for the above inequalities that $m=n_2-n_1$.)
\end{proof}

As a consequence of the lemma, the probability of mismatch after
$m$ matches is dominated by
\begin{equation}
\gamma_m:=2(e^{F_m}-1),
\end{equation}

Then the probability that we are not coupled at time $k=l$ (i.e.,
the spins at the origin in the coupling are unequal) can be
estimated by
\begin{eqnarray} \nonumber
\left|\mu_{[-l,l],\zeta}(\sigma_0)-\mu_{[-l,l],\zeta'}(\sigma_0)\right|=
\left|\mathbb{E}_{\mathbb{P}_{12}}(\sigma_0^1-\sigma_0^2)\right|
\end{eqnarray}
where $\mathbb{P}_{12}$ denotes the coupling of the measures
$\mu_{[-l,l],\zeta}(\sigma_{[-l,l]}=\cdot)$ and
$\mu_{[-l,l],\zeta'}(\sigma_{[-l,l]}=\cdot)$ just described.

Remark that by the non-nulness of Gibbs measures, we have that

\[
\sup_{\alpha, \beta,\zeta,\xi,\zeta_{\Delta_m}=\xi_{\Delta_m}}
\mu_{[-n_1,n_1],\zeta}(\sigma_{-n_1}=\alpha,\sigma_{n_1}=\beta) < 1-
\delta
\]
for some $0<\delta<1$.
As in
\cite{Bressaud99}, we then consider the auxiliary Markov chain $S_n$ on
$\{0,1,2,\cdots\}$ whose transition probabilities are
\begin{equation}
\begin{cases}
\mathbb{P}(S_{n+1}=m+1|S_n=m)=1-\min \{ \gamma_m, 1-\delta\} \\
\mathbb{P}(S_{n+1}=0|S_n=m)=\min \{ \gamma_m, 1-\delta\}.
\end{cases}
\end{equation}
On the other hand, we have the process that counts the number of
matches (the so-called "house-of-cards" process), defined by
\begin{equation} \label{pig}
\begin{cases}
Z_0 = 0\\
Z_{n+1} =
\begin{cases} Z_{n} + 1 \;\;\; \mbox{if} \;\;
(\sigma_{-l+n}^1,\sigma_{l-n}^1)=(\sigma_{-l+n}^2,\sigma_{l-n}^2)\\
0 \;\;\;\;\;\;\;\;\;\;\;\; \mbox{otherwise}
\end{cases} \mbox{for} \;\; n=0,1,2,\ldots
\end{cases}
\end{equation}
By Proposition 1 in \cite{Bressaud99}, we have
\begin{equation}
\left|\mu_{[-l,l],\zeta}(\sigma_0)-\mu_{[-l,l],\zeta'}(\sigma_0)\right|
= \left|\mathbb{E}_{\mathbb{P}_{12}}(\sigma_0^1-\sigma_0^2)\right| =
\mathbb{P}(Z_{l}=0)\leq\mathbb{P}(S_{l}=0).
\end{equation}
Finally condition (\ref{potential-assumption2}) insures that
$\gamma_n\rightarrow 0$ as $n\rightarrow +\infty$. Then by
Proposition 2 in \cite{Bressaud99}, we have
$\mathbb{P}(S_{l}=0)\rightarrow 0$ as $l\rightarrow +\infty$, which
completes the proof.
\end{proof}

\subsection{The transformed potential}\label{4}
\begin{defn} If $\mu$ is a measure that admits a continuous version
of the conditional probabilities
$\mu(\xi_i|\xi_{\Z\backslash\{i\}})$, $i\in \Z$, then we call
$\varphi$ an estimate for the rate of continuity if
\begin{equation}\label{pincemon}
\sup_{\xi,\zeta}\left|\mu(\xi_i|\xi_{[-n,n]\backslash\{i\}}\zeta_{[-n,n]^C}) -
\mu(\xi_i|\xi_{\Z\backslash\{i\}})\right| \leq \varphi(n).
\end{equation}
\end{defn}

In the previous section we showed that for our transformed Gibbs measure,
$\mathbb{P}(S_n=0)$ is an estimate for the rate of continuity. We
now show the decay of the Kozlov potential associated to $\mu$, when
we have an estimate on the rate of continuity. We start from the
following explicit form of the potential of theorem \ref{Kozlov-Sullivan}, see \cite{Kozlov74},
\cite{MRV}.
We assume, without loss of generality, that the finite
alphabet contains a distinguished symbol denoted by $``+''$.
\begin{thm}\label{papapa}
Let $\nu$ be a probability measure such that the conditional
probabilities $\nu(\xi_i|\xi_{\Z\backslash\{i\}})$, $i\in \Z$, are
non null and have a continuous version. Consider the potential,
defined on lattice intervals (and vanishing on other subsets) by
\begin{equation}
U([i,j],\xi)=\log
\frac{\nu(\xi_i|\xi_{]i,j[}+)\nu(\xi_j|\xi_{]i,j[}+)}
{\nu(\xi_i\xi_j|\xi_{]i,j[}+)},
\end{equation}
where the plus signs mean that conditioned sites outside the lattice
interval $[i,j]$ all have the state $+$. If $U$ is
uniformly absolutely convergent, then $\nu$ is a Gibbs measure
associated with the potential $U$.
\end{thm}

We look now at this potential in our context, i.e., when $\nu$ is
the transformed Gibbs measure $\mu\circ T$. By Theorem
\ref{thoerem1}, $\mathbb{P}(S_n=0)$ is an estimate for the rate of
continuity of $\mu\circ T$. We can then estimate the potential: if
$\varphi$ is an estimate for the rate of continuity,
\begin{eqnarray} \nonumber
\nu(\xi_i|\xi_{]i,j[}+) &=&
\nu(\xi_i|\xi_{]i,j]}+)\nu(\xi_j|\xi_{]i,j[}+) + \sum_{\eta_j\not=
\xi_j}\nu(\xi_i|\xi_{]i,j[}\eta_j+)\nu(\eta_j|\xi_{]i,j[}+)
\\ &\leq& \nu(\xi_i|\xi_{]i,j]}+)\left(1+C\varphi(|j-i|)\right),
\end{eqnarray}
where the constant $C$ is bounded by the non-nullness assumption.
Further, we have
\begin{equation} \nonumber
\nu(\xi_i\xi_j|\xi_{]i,j[}+) =
\nu(\xi_i|\xi_{]i,j]}+)\nu(\xi_j|\xi_{[i,j[}+).
\end{equation}
So we have the estimate on Kozlov potential of the transformed
measure
\begin{equation}\label{vampier}
\left|U_{[i,j]}(\xi)\right| \leq \log (1+C\varphi(|j-i|)) \leq
C\varphi(|j-i|)
\end{equation}

We now consider two relevant cases, according to behavior of
$F_m$ in \eqref{aam}.

\begin{enumerate}
\item If $f$ in \eqref{potential-assumption} decays {\em exponentially}, then
$F_N$ decays also exponentially as $N$ increases.
This implies that $\varphi$ in \eqref{vampier} also decays exponentially, that is,
\begin{equation}
U_{[i,j]}(\xi) \leq e^{-\lambda|j-i|}
\end{equation}
for some $\lambda >0$.
\item  In case that $f$ decays as {\em a power law}
i.e., for some $C>0$,
\begin{equation}
f(k) \leq \frac{C}{k^{\alpha}},\;\;\; \mbox{for}\;\alpha>1,
\end{equation}
we have
\begin{equation}
F_N\leq \frac{C_1}{N^{\alpha-1}},
\end{equation}
where $C_1$ is a positive constant. This implies that $\varphi$
in \eqref{pincemon}
decays as $\frac{C_1}{n^{\alpha-1}}$, which in turn implies that the
transformed potential decays as
\begin{equation}
\|U_{[i,j]}\|_{\infty} \leq \frac{C_1}{(j-i)^{\alpha-1}}.
\end{equation}
Hence, $\alpha >2$ is sufficient to have uniform absolute
summability of this potential (whereas $\alpha >1$ is sufficient for
Gibbsianness of the transformed measure)

E.g., if the original
potential is a long-range Ising potential, i.e.,

\[
\Phi(\{i,j\},\sigma)=\frac{\sigma_i\sigma_j}{|j-i|^{\gamma}}
 \]
then
we need $\gamma>2$ for the transformed measure to be Gibbsian, and
$\gamma>3$ for the transformed potential to be uniformly absolute
convergent. Remark that for $\gamma <2$
we do not have uniqueness of the associated Gibbs measure, so
the transformed measure might be non-Gibbsian.
\end{enumerate}
\section{Deterministic single site transformations}\label{5}
As before, we consider the configuration space of the untransformed
system $\Omega= S^{\Z}$, where $S$ is a finite set, and the
configuration space of the transformed system is $\Omega'=
(S')^{\Z}$. The transformation $T:\Omega\to\Omega'$ now is induced
by a map $\vi: S\to S'$, via \be\label{dettra} (T(\si))_i=:\si'_i =
\vi (\si_i) \ee This is equivalent with defining the new spin $\si_i$
via a partition of the single-site space $S$, which in the case of
$S= \{ 1,\ldots, q\}$ and $\Phi$ the potential of the Potts-model
has been called the fuzzy Potts model, see \cite{mvdv}.

To deal with such transformations, we follow the approach of in
\cite{opoku}. This consists of writing the single-site conditional
probabilities of the transformed measure in terms of a so-called
{\em constrained restricted first layer measure}. The difference
with stochastic transformations is that this measure does not
necessarily have full support, i.e., given the second layer
constraint $\xi\in\Omega'$, the first layer has to be such that its
image coincides with $\xi'$.

As in the previous section, we start with a Gibbs measure $\mu$ on
configurations $\si\in \Omega$. The potential $\Phi$ satisfies
\eqref{potential-assumption2}. We further abbreviate $\nu = \mu\circ
T$ and $K(\eta_i|\si_i) = I(\vi (\si_i ) =\eta_i)$, where $I$
denotes indicator, and for $\la\subset\Z$ finite, $\la_0
:=\la\setminus \{ 0 \}$.

For clarity, we first repeat the main steps of \cite{opoku} to rewrite the
single-site conditional probabilities of $\nu$ in terms of a
constrained restricted first layer measure.
\beq \nu
(\eta_0|\eta_{\la_0}) &=& \frac{\sum_{\si_\la} \mu(\si_\la)
\prod_{i\in\la} K(\eta_i|\si_i)} {\sum_{\si_\la} \mu(\si_\la)
\prod_{i\in\la_0} K(\eta_i|\si_i)}
\nonumber\\
&=& \frac{\int \mu(d\zeta)\sum_{\si_{\la}} \mu_{\la,\zeta}(\si_\la)
\prod_{i\in\la} K(\eta_i|\si_i)} {\int \mu(d\zeta)\sum_{\si_\la}
\mu_{\la,\zeta}(\si_\la) \prod_{i\in\la_0} K(\eta_i|\si_i)} \eeq Now
we consider the following auxiliary measure on the state space
$\Omega_0:= S^{\la_0}$.

\be\label{rflm} \mu^{\eta_{\la_0}}_{\la_0,\zeta}(\si_{\la_0}):=
\frac{1}{N^\eta_{\la,\zeta}}\exp \left(- \caH_{\la_0}^\zeta
(\si_{\la_0})\right) \prod_{i\in\la_0} K(\eta_i|\si_i) \ee where
$N^\eta_{\la,\zeta}$ denotes the normalizing constant, and
\be\label{ho} \caH_{\la_0}^\zeta (\si_{\la_0})=\sum_{A\cap
\la_0\not= \emptyset, A \not\ni 0} \Phi (A, \si_\la\zeta_{\la^c})
\ee These measures concentrate on configurations $\si_{\la_0}\in
S^{\la_0}$ compatible with $\eta_{\la_0}$, i.e., such that
$K(\eta_i|\si_i)\not= 0$ for all $i\in\la_0$. For $\eta\in\Omega'$
fixed, they form a $\eta$-dependent specification on the
configuration space $S^{\Z_0}$, i.e.,
\begin{itemize}\label{speci}
 \item[a)] $\mu^{\eta_{\la_0}}_{\la_0,\zeta}(\si_{\la_0})$ is a probability measure
on $S^{\la_0}$
\item[b)]$ \mu^{\eta_{\la_0}}_{\la_0,\zeta}(\si_{\la_0})$ depends only
in $\zeta$ on $\Z_0\setminus \la_0$
\item[c)] Consistency: if we denote
\be\label{hoeraa}
 \left(\gamma_{\la_0}^\eta (g)\right)(\zeta) := \int \mu^{\eta_{\la_0}}_{\la_0,\zeta}(d\si_{\la_0}) g(\si_{\la_0} \zeta_{\la^c})
\ee then these $\eta$-dependent kernels $\gamma_\la^\eta$ satisfy
\be\label{cons}
 \gamma^\eta_{\la_0}(\gamma^\eta_{\la'_0}(g)) =
 \gamma^\eta_{\la_0}(g)
\ee for all $\la\supset\la'$ and all local functions $g$.
\end{itemize}

In terms of these measures, we can rewrite the conditional
probability $\nu (\eta_0|\eta_{\la_0})$ as follows. \be \nu
(\eta_0|\eta_{\la_0})= \frac{\int \mu(d\zeta)\frac{
N^\eta_{\la,\zeta}}{Z_\la^\zeta}\int \mu^{\eta_{\la_0}}_{\la_0,
\zeta} (d\si_{\la_0}) \psi_{0,\la}^\zeta (\eta_0,\si_{\la_0})} {\int
\mu(d\zeta) \frac{ N^\eta_{\la,\zeta} }{Z_\la^\zeta}\int
\mu^{\eta_{\la_0}}_{\la_0, \zeta} (d\si_{\la_0}) \vi_{0,\la}^\zeta
(\si_{\la_0})} \ee where \beq\label{phio} \psi_{0,\la}^\zeta
(\eta_0,\si_{\la_0}) &=& \sum_{\si_0} e^{-h_0 (\si_0\si_{\la_0}
\zeta_{\la^c})} K(\eta_0|\si_0)
\nonumber\\
\vi_{0,\la}^\zeta (\si_{\la_0}) &=& \sum_{\si_0} e^{-h_0
(\si_0\si_{\la_0}) \zeta_{\la^c})} \eeq with \be\label{haao} h_0
(\si) = \sum_{A\ni 0} \Phi (A,\si) \ee and $Z_\la^\zeta$ is the
finite-volume partition function with boundary condition $\zeta$,
i.e.,
\[
Z_\la^\zeta= \sum_{\si_\la} e^{-H_\la^\zeta(\si_\la)}
\]
Notice that $\psi_{0,\la}^\zeta(\eta_0,\sigma_{\Lambda_0})$ and
$\vi_{0,\la}^\zeta(\sigma_{\Lambda_0})$ converge uniformly (in
$\eta_0, \sigma, \zeta$), as $\la\uparrow\Z$ to \beq\label{phioo}
\psi_{0}(\eta_0,\si_{\Z\setminus \{0\}}) &=& \sum_{\si_0} e^{-h_0
(\si)} K(\eta_0|\si_0)
\nonumber\\
\vi_{0}(\si_{\Z\setminus\{0\}}) &=& \sum_{\si_0} e^{-h_0 (\si)} \eeq

\subsection{Exponentially decaying potential}
Let us now first look at the case where $\Phi$ decays exponentially.
As a consequence, the decay to zero in \eqref{potential-assumption2}
is exponential in $k$. We will prove here, that, as in the
stochastic case, the transformed measure $\nu$ has an exponentially
decaying interaction as well. In this case, for $\la= [-n,n]$ there
exist $C_1, c_1>0$ such that for all $\zeta, \si, \eta$,
\[
 |\psi_{0,\la}^\zeta (\eta_0,\si_{\la_0})-\psi_{0}(\eta_0,\si_{\Z\setminus\{0\}})|\leq C_1 e^{-c_1 n}
\]
and similarly for $\vi_0$. Our aim is then to show that there exist
$C_2,c_2>0$ such that for all $\eta$, $n, m>n$,
\[
 |\nu(\eta_0|\eta_{[-n,n]_0}) -\nu(\eta_0|\eta_{[-m.m]_0})|\leq C_2 e^{-c_2 n}
\]
The idea is once more to couple the measures
$\mu^{\eta_{\la_0}}_{\la,\zeta}$ and
$\mu^{\eta_{\la_0}}_{\la,\zeta'}$ for different boundary conditions,
such that in the coupling the probability that $\si^1_i\not=
\si^2_i$ is bounded by $e^{-\alpha|n-i|\wedge|-n-i|}$ for some
$\alpha>0$. This coupling follows the same iterative procedure as in
the stochastic case, and the estimates are identical. Next, we need
to compare expectations of the functions $\psi_0, \vi_0$ (instead of
a function that only depends on $\si_0$ in the stochastic case).
These functions $\psi_0, \vi_0$ can however be exponentially well
approximated by local functions. We spell out these steps in three
lemmas.

\begin{lem}\label{bam1} Let $\mu_1,\mu_2$ be two probability
measures on $S^{\la_0}$ and $\pee$ a coupling of them. Then for all
functions $g:S^{\la_0}\to\R$ we have \be\label{cocooo} \left|\int g\
d\mu_1-\int g\ d \mu_2\right| \leq \sum_{i\in\la_0} \pee
(\si^1_i\not=\si^2_i) \delta_i g \ee where $\delta_i g( \si) =
\sup\{ g(\si)-g(\si'): \si_j=\si'_j \ \forall\ j\not= i\}$ \end{lem}
\bpr This is elementary and left to the reader. \epr

\begin{lem}\label{bam2} For $\la= [-n,n]$ there exists a coupling $\pee$ of
$\mu^{\eta_{\la_0}}_{\la,\zeta}$ and
$\mu^{\eta_{\la_0}}_{\la,\zeta'}$ such that \be\label{cocaaa}
 \pee (\si^1_i\not=\si^2_i) \leq C_3 e^{-c_3 |n-i|\wedge|-n-i|}
\ee

where $C_3, c_3>0$ do not depend on $\zeta,\zeta', n$. \end{lem}
\bpr The coupling follows the iterative procedure as in  the
stochastic case, and the estimates in terms of the function $f$ in
\eqref{potential-assumption} are identical. \epr

As a consequence of
these lemmas we have the existence of a unique Gibbs measure
$\mu^\eta$ on $S^{\Z_0}$ consistent with the specification
$\mu^{\eta_{\la_0}}_{\la,\zeta}$, and for any local function $g$
(with dependence set in $\Lambda$) we have the estimate

\be\label{bambam}\sup_{\xi} \left|\int g d\mu^\eta - \gamma^\eta_\la (\xi)
(g)\right|\leq C_3 \sum_{i} \delta_i( g)  e^{-c_3 |n-i|\wedge|-n-i|}
\ee
where we used the notation \eqref{hoeraa} and where $\la =
[-n,n]$

\begin{lem}\label{bam3} Suppose that $g:S^{\Z_0}\to\R$ is continuous and such
that there exist $g_k$ depending only on $\si_i, i\in [-k,k]_0$ such
that \be\label{goror}
 \| g_k -g\|_\infty < C_4 e^{-c_4 k}
\ee for some $C_4, c_4>0$. Then there exists $C_5, c_5 >0$ such that
for $\la= [-n,n]$ \be\label{domo} \sup_{\eta}\left|\gamma^\eta_\la
(g) -\int g\ d\mu^\eta\right|\leq C_5 e^{-c_5 n} \ee \end{lem} \bpr
Choose $\la= [-n,n]$, and choose $g_k$ as in \eqref{goror}. Write
\beq |\gamma_{\la}^\eta (g ) (\zeta) -\gamma_\la^\eta (g) (\xi)|
\leq A+B+C \eeq where \beq A:= |\gamma_{\la}^\eta (g ) (\zeta)
-\gamma_\la^\eta (g_k) (\zeta)| \leq \| g-g_k\|_\infty \eeq

\beq C:= |\gamma_{\la}^\eta (g_k ) (\xi) -\gamma_\la^\eta (g) (\xi)|
\leq \| g-g_k\|_\infty \eeq

 \beq B:= |\gamma_{\la}^\eta (g_k )
(\zeta) -\gamma_\la^\eta (g_k) (\xi)| \leq  2\sup_{\xi} \left|
\gamma_\la^\eta (g_k) (\xi)-\int g_k d\mu^\eta\right|
\eeq

Now use \eqref{bambam}, and the obvious inequality $\delta_i (g)
\leq 2 \|g\|_\infty$ to obtain \beq |\gamma_{\la}^\eta (g ) (\zeta)
-\gamma_\la^\eta (g) (\xi)| \leq 2 C_4 e^{-c_4 k} +
4\sup_k\|g_k\|_\infty\sum_{j=0}^k C_3 e^{-c_3 (n-j)} \eeq Finally,
choose $k= n/2$. \epr

\subsection{Power law decaying potential}
For the case where $\Phi$ decays according to a power law, more precisely, if
\be\label{powerlaaw}
f(K) \leq C k^{-\alpha}
\ee
where $f$ is the function associated to the potential
$\Phi$  as in \eqref{potential-assumption}, and $\alpha >2$.
Then we have the analogue of \eqref{cocaaa}
(cf.\ the two cases considered after Theorem \ref{papapa})
\be\label{cocaaaa}
 \pee (\si^1_i\not=\si^2_i) \leq C_3 \left((n-i) \wedge (-n-i)\right)^{\alpha-1}
\ee
Next, the local approximations of the functions
$\psi_0$ and $\vi_0$ converge now only at power-law speed, i.e.,
the local approximations
$\psi_0^k$, $\vi_0^k$
with dependence set $[-k,k]$ satisfy
\[
 \|\psi_0-\psi_0^k\|_\infty < C k^{-\alpha}, \ \|\vi_0-\vi_0^k\|_\infty < C k^{-\alpha}
\]
Therefore, in that case we find, using the same steps
as in the exponential case, for all
$\eta$, $n, m>n$,
\[
 |\nu(\eta_0|\eta_{[-n,n]_0}) -\nu(\eta_0|\eta_{[-m.m]_0})|\leq C_2 n^{-(\alpha -2)}
\]

\end{document}